\documentclass[11pt]{amsart}
\usepackage{amssymb,amsmath,url,graphicx, amscd}
\usepackage{hyperref}
\usepackage{stmaryrd} 
\usepackage{caption}
\usepackage{enumerate}
\usepackage{subcaption}
\usepackage{courier}
\usepackage[all]{xy}
\usepackage{tikz,tikz-3dplot}
\usepackage[normalem]{ulem}
\usepackage{color,soul}
\newcommand{\CC}{\mathbb{C}}

\newtheorem {Theorem}[equation]         {Theorem}

\newtheorem {Lemma}[equation]           {Lemma}

\newtheorem {Claim*}                    {Claim}

\newtheorem {Corollary} [equation]      {Corollary}
\newtheorem {Proposition}  [equation]   {Proposition}

\theoremstyle{definition}

\theoremstyle{remark}
\newtheorem{Remark}[equation]{Remark}
\newtheorem*{Remark*}{Remark}

\title{The geometric structure of symplectic contraction}
\author{Jeremy Lane}
\address{Dept.\ of Mathematics, Universit\'e de Gen\`{e}ve, 2-4 rue du Li\`{e}vre, Case postale 64
1211 Gen\`{e}ve 4, Switzerland}
\email{jeremy.lane@unige.ch}

\date{\today}

\keywords{hamiltonian group actions, symplectic contraction, horospherical contraction, gelfand-zeitlin}

\begin{document}
\maketitle

\begin{abstract}
	We show that the symplectic contraction map of Hilgert-Manon-Martens \cite{hmm} -- a symplectic version of Popov's horospherical contraction -- is simply the quotient of a Hamiltonian manifold $M$ by a ``stratified null foliation'' that is determined by the group action and moment map. We also show that the quotient differential structure on the symplectic contraction of $M$ supports a Poisson bracket. We end by proving a very general description of the topology of fibers of Gelfand-Zeitlin systems on multiplicity free Hamiltonian $U(n)$ and $SO(n)$ manifolds.
\end{abstract}

\section{Introduction}

Degenerations and their gradient-Hamiltonian flows are a major theme in recent studies of  interactions between algebraic geometry, representation theory, and symplectic geometry. Although it can be difficult to precisely describe the gradient-Hamiltonian flow of a given degeneration -- for the simple reason that the defining differential equation can be quite complicated -- an improved understanding the geometry of this flow is desirable, since it often leads to interesting new results lying at the interface between algebraic and symplectic geometry (cf. \cite{nnu,harada-kaveh,hp,flp,kaveh}).

Recent work by Hilgert-Manon-Martens (HMM) provides an algebraic formula for the time-1 flow of Popov's degeneration of a semi-projective variety equipped with an action by a connected complex reductive group $G$ \cite{pop}, to its horospherical contraction \cite{hmm}. To this end, HMM define more generally, the \emph{symplectic contraction}, $M^{sc}$, of any Hamiltonian $K$-manifold $(M,\omega,\mu)$ ($K$ compact, connected) and the \emph{symplectic contraction map}, $\Phi_M\colon M \to M^{sc}$. HMM prove that if $M$ is semi-projective and $G = K^{\CC}$, then $\Phi_M$ coincides with the time-1 flow of the gradient-Hamiltonian flow of Popov's degeneration \cite[Corollary 5.12]{hmm}.
 
Although a formula for $\Phi_M$ -- and thus, when everything is sufficiently algebraic, a formula for the time-1 flow of horospherical degeneration -- presents significant progress, both the definition of $M^{sc}$ as a diagonal reduction of a product of symplectic imploded spaces, and the formula for $\Phi_M$, are somewhat opaque from the perspective of symplectic geometry.

In this note we study the geometry of symplectic contraction in more detail. We observe that: 
\begin{enumerate}[i)]
	\item there is a naturally defined decomposition of any Hamiltonian $K$ manifold $M$ into coisotropic submanifolds (whose definition only depends on the action of $K$ and the moment map), 
	\item the quotient of $M$ by the null foliation of these coisotropic submanifolds is isomorphic to $M^{sc}$ (i.e. there is a stratification preserving $K\times T$ equivariant homeomorphism of the two spaces whose restriction to the symplectic strata is a symplectomorphism), and 
	\item with this identification, $\Phi_M\colon M \to M^{sc}$ is simply the quotient map for the stratified null foliation of $M$.
\end{enumerate}

In addition to demonstrating symplectic contraction as a natural geometric quotient of a Hamiltonian manifold, in many ways analogous to Marsden and Weinstein's symplectic reduction and Guillemin-Jeffreys-Sjamaar's symplectic implosion, this perspective has several immediate consequences. First, it is obvious that the restriction of the map $\Phi_M$ to the open dense piece of $M$ is a symplectomorphism onto its image: the null foliation of a symplectic manifold is trivial! Second, from this perspective one observes that $M^{sc}$ has a naturally defined Poisson algebra of smooth functions, which endows it with the structure of a symplectic stratified space in the sense of \cite{sl}. Finally, from this perspective we see that the symplectic contraction map (and in the algebraic case, the time-1 gradient-Hamiltonian flow of horospherical degeneration) is not just a continuous map that extends a smooth map on an open dense set: it is smooth in a stratified sense and Poisson in a differential sense.

In Section 2 and 3 we describe the results outlined above. In Section 4 we discuss this geometric perspective in relation to branching contraction (iterated symplectic contraction) and Gelfand-Zeitlin systems. In particular, we prove that the symplectic pieces of the branching contraction corresponding to a Gelfand-Zeitlin system (on an arbitrary multiplicity free $U(n)$ or $SO(n)$ manifold) are all toric manifolds (Theorem \ref{pieces-are-toric}). It follows from this fact that the fibers of Gelfand-Zeitlin systems can be described geometrically as iterated bundles of coisotropic homogeneous spaces over isotropic tori.  This provides a more general version of results obtained for Gelfand-Zeitlin systems on $U(n)$ coadjoint orbits by \cite{cko}, which only came to the attention of the author of this paper after completing this paper. 

\textbf{Acknowledgements:} The author would like to thank Yael Karshon, Megumi Harada, Kuimars Kaveh, and Chris Manon for  discussions and feedback. This work was supported by a NSERC PGSD award and NCCR SwissMAP of the Swiss National Science Foundation.

\section{The ``stratified null foliation'' of a Hamiltonian $K$-manifold}\label{s:2}

Let $K$ be a compact, connected Lie group with Lie algebra $\mathfrak{k}$. In this section we show that any Hamiltonian $K$-manifold $(M,\omega,\mu)$ has a ``stratified null foliation'' determined by the action of $K$ and the moment map $\mu$ and describe the quotient of $M$ by this foliation. The stratified null foliation of $M$ is closely related to its symplectic implosion and, as a result, this section is inspired by many ideas from \cite{gsj}.

If one fixes a maximal torus $T$ and a positive Weyl chamber $\Delta \subseteq \mathfrak{t}^*$, there is a maximal slice for the coadjoint action of $K$ at each stratum\footnote{As a polyhedral cone, the positive Weyl chamber $\mathfrak{t}_+^* = \Delta$ has a natural stratification by relative interiors of faces.} $\sigma \subseteq \Delta$ given by $\mathfrak{S}_{\sigma} : = K_{\sigma}\cdot \text{star}(\sigma) \subseteq \mathfrak{k}_{\sigma}^*$, where $K_{\sigma}$ is the stabilizer subgroup of points in $\sigma$ for the coadjoint action of $K$ (cf. \cite{gsj}).  For a Hamiltonian $K$-manifold $(M,\omega,\mu)$, let $M_{\sigma} := \mu^{-1}(\mathfrak{S}_{\sigma})$ denote the symplectic cross-section of $M$ at $\sigma$.  We recall the symplectic cross-section theorem, as stated in \cite[Theorem 2.5]{gsj}, which we will use below.

\begin{Theorem}[The symplectic cross-section theorem]\label{th:symplectic slice theorem} Let $(M,\omega,\mu)$ be a Hamiltonian $K$-manifold. Then,
\begin{enumerate}
	\item $M_{\sigma}$ is a $K_{\sigma}$-invariant symplectic submanifold of $M$ and the restriction of $\mu$ to $M_{\sigma}$ is a moment map for the $K_{\sigma}$ action.
	\item The map $K\times M_{\sigma} \to M$ given by $(k,m) \mapsto k\cdot m$ induces a symplectomorphism $K\times_{K_{\sigma}}M_{\sigma} \cong KM_{\sigma}$ (with respect to the symplectic structure on $K\times_{K_{\sigma}}M_{\sigma}$ that will be described below) onto its image, which is open and dense in $M$.
	\item If $\sigma_{prin}$ is the principal stratum\footnote{Recall that the principal stratum of $\Delta$ corresponding to a Hamiltonian $K$-manifold $M$ is the unique stratum such that $\mu(M)\cap \sigma$ is dense in $\mu(M) \cap \Delta$.} of $\Delta$ corresponding to $(M,\omega,\mu)$, then $ [K_{\sigma_{\text{\emph{prin}}}},K_{\sigma_{\text{\emph{prin}}}}]$ acts trivially on $M_{\sigma_{\text{\emph{prin}}}}$.
\end{enumerate}
\end{Theorem}

By part (1) of the symplectic cross-section theorem, the action of $K_{\sigma}' = [K_{\sigma},K_{\sigma}]$ on $M_{\sigma}$ is Hamiltonian, generated by $p \circ \mu$ where $p\colon \mathfrak{k}_{\sigma}^* \to (\mathfrak{k}_{\sigma}')^*$ is the dual projection map. Following \cite{gsj}, observe that the zero level set of this moment map 
	$$(p\circ \mu)^{-1}(0) = \mu^{-1}(\sigma).$$
	It follows by \cite[Theorem 2.1]{sl}, that for every closed subgroup $H\leq K_{\sigma}'$ the intersection of $(p\circ \mu)^{-1}(0)$ with the orbit-type stratum of $H$ for the Hamiltonian action of $K_{\sigma}'$ on $M_{\sigma}$ is a coisotropic submanifold,
$$ (p\circ\mu|_{M_{\sigma}})^{-1}(0)\cap M_{\sigma,(H)} = \mu^{-1}(\sigma) \cap M_{\sigma,(H)} \subseteq M_{\sigma}$$
and, moreover, the leaves of the null foliation of this coisotropic submanifold equal the orbits of the action of $K_{\sigma}'$. We denote this coisotropic submanifold $Q_{\sigma,(H)}$ (note that $Q_{\sigma,(H)}$ may have multiple connected components of varying dimension). 

Define 
\begin{equation}\label{definition-of-W}
	W_{\sigma,(H)} := K\cdot Q_{\sigma,(H)}.
\end{equation}
Note that 
	$$W_{\sigma,(H)} = \left\{ m \in \mu^{-1}(\Sigma_{\sigma}) \colon\, \text{Stab}_{K_{\mu(m)}'}(m) \in (H) \right\}$$
where $(H)$ denotes the conjugacy class of $H$ in $K$.  By part (2) of the symplectic cross-section theorem (cf. \cite[Theorem 41.1]{gs})
$$W_{\sigma,(H)} \cong K \times_{K_{\sigma}}Q_{\sigma,(H)} \subseteq K \times_{K_{\sigma}}M_{\sigma} \cong (K \times \mathfrak{S}_{\sigma} \times M_{\sigma})\sslash_0 K_{\sigma}$$
where the space on the right is equipped with the symplectic form $\Omega|_{K \times \mathfrak{S}_{\sigma}} + \omega|_{M_{\sigma}}$ and the reduction by $K_{\sigma}$ is diagonal, generated by the moment map $\mu_{\mathcal{R}}+\mu|_{M_{\sigma}}$ (here $\Omega$ is the canonical symplectic form on $T^*K \cong K\times \mathfrak{k}^*$ and $\mu_{\mathcal{R}}(k,\lambda) = \lambda$ is the moment map for the cotangent lift of the right action of $K$ on itself).

\begin{Proposition}\label{prop:2}
	Each $W_{\sigma,(H)}$ is a coisotropic submanifold of $M$.  For every $m\in W_{\sigma,(H)}$, the leaf of the null foliation of $W_{\sigma,(H)}$ through $m$ equals the orbit $K_{\mu(m)}'\cdot m$.
\end{Proposition}

\begin{proof}
	We first show that for $q\in Q_{\sigma,(H)}$, the leaf of the null foliation of $W_{\sigma,(H)}$ through $q$ equals the orbit $K_{\sigma}'\cdot q$.
	
	Every element of $TqW_{\sigma,(H)}$ can be represented by $\underline{Y}+X$ for $Y \in \mathfrak{k}$ and $X\in TQ_{\sigma,(H)}.$ Every element of $TqM_{\sigma}$ can be represented by $\underline{Y'}+Z$ for $Y' \in \mathfrak{k}$ and $Z\in TM_{\sigma}.$
	
	If for all $Y\in \mathfrak{k}$,
	$$0 = \omega_q(\underline{Y},\underline{Y'}+Z) = \langle d\mu_q(\underline{Y'}+Z),Y\rangle$$
	then it must be true that $\underline{Y'}+Z \in \ker(d\mu_q)\subseteq T_qM_{\sigma}$, in which case $\underline{Y'}+Z$ can be represented by $Z\in T_qM_{\sigma}$.
	
	If for all $X\in T_qQ_{\sigma,(H)}$, 
	$$0 = \omega_q(X,Z),$$
	then it must be the case that $Z\in (TqQ_{\sigma,(H)})^{\omega}$, which by \cite[Theorem 2.1]{sl} equals $T_q(K_{\sigma}'\cdot q)$.
	
	For $m = k\cdot q$, the result follows since $K$ acts by symplectomorphisms: the leaf of the null foliation of $W_{\sigma,(H)}$ through $k\cdot q$ is the set 
	$$k\cdot(K_{\sigma}'\cdot q) = (kK_{\sigma}'k^{-1})\cdot (k\cdot q) = K_{\mu(m)}'\cdot m.$$	
\end{proof}

Since $W_{\sigma,(H)}$ is coisotropic, there is a quotient map 
$$\pi \colon W_{\sigma,(H)} \to W_{\sigma,(H)}/_{\sim}$$
defined by the null foliation of $W_{\sigma,(H)}$. However, unlike the situation of symplectic reduction, the leaves of this null foliation do not equal the orbits of a compact group action. Instead, observe the following obvious identifications,
\begin{equation}\label{diffeomorphism}
\left( K \times_{K_{\sigma}}Q_{\sigma,(H)}\right)/_{\sim} \cong_{\text{Homeo}} K\times_{K_{\sigma}}\left(Q_{\sigma,(H)}/K_{\sigma}'\right) \cong_{\text{Diff}} K/K_{\sigma}' \times_{Z_{\sigma}} \left(Q_{\sigma,(H)}/K_{\sigma}'\right).
\end{equation}
Under the homeomorphism on the left, the quotient map $W_{\sigma,(H)} \to W_{\sigma,(H)}/_{\sim}$ is identified with the smooth submersion
$$K \times_{K_{\sigma}}Q_{\sigma,(H)} \to K\times_{K_{\sigma}}\left(Q_{\sigma,(H)}/K_{\sigma}'\right),\, (k,q) \mapsto (k,[q]).$$
Thus we conclude the following (cf. \cite[Theorem 25.2]{gs}).

\begin{Proposition}
	The quotient of each $W_{\sigma,(H)}$ by its null foliation is a smooth manifold equipped with a symplectic form, which we  denote $\widetilde\omega$, defined by the property that 
	$$ \pi^*\widetilde\omega = \omega|_{W_{\sigma,(H)}}.$$
\end{Proposition}

Combining the results above, we have decomposed $M$ as a disjoint union
$$M = \bigcup_{\sigma\subseteq \Delta}\bigcup_{H\leq K_{\sigma}'}W_{\sigma,(H)}$$
such that each $W_{\sigma,(H)}$ is a smooth, coisotropic submanifold of $M$, invariant under the action of $K$. We define an equivalence relation on $M$ by $m\sim m'$ if $m,m'$ are contained in the same leaf of the null foliation of one of the pieces $W_{\sigma,(H)}$. We call this the \emph{stratified null foliation} of $(M,\omega,\mu)$. 

The quotient of $M$ by the stratified null foliation is a topological space with a decomposition into pieces that are smooth symplectic manifolds:
$$ M/_{\sim}\, = \bigcup_{\sigma\subseteq \Delta}\bigcup_{H\leq K_{\sigma}'}W_{\sigma,(H)}/_{\sim}.$$
Although it is not a manifold, $M/_{\sim}$ has a naturally defined algebra of functions (the quotient differential structure),
$$C^{\infty}(M/_{\sim}) : = \left\{ f \in C(M/_{\sim}) \colon\, \pi^*f \in C^{\infty}(M)^{\sim}\right\}$$
where $\pi\colon M \to M/_{\sim}$ is the quotient map and $C^{\infty}(M)^{\sim}\subseteq C^{\infty}(M)$ is the subalgebra of smooth functions on $M$ that are locally constant on leaves of the stratified null foliation.  The inclusion maps $\iota\colon W_{\sigma,(H)}/_{\sim} \to M/_{\sim}$ are smooth in the differential sense: for all $f\in C^{\infty}(M/_{\sim})$, 
$$\pi^*(\iota^*f) = (\pi^*f)|_{W_{\sigma,(H)}}$$
so $\iota^*f \in C^{\infty}(W_{\sigma,(H)}/_{\sim})$.  


\begin{Proposition}\label{poisson-structure} The bracket on $C^{\infty}(M/_{\sim})$ defined by the equation
$$ \pi^*\{f,g\}_{M/\sim} = \{\pi^*f,\pi^*g\}_M$$ is a Poisson bracket. Moreover, the inclusion maps $\iota\colon W_{\sigma,(H)}/_{\sim} \to M/_{\sim}$ are Poisson with respect to $\{\cdot,\cdot\}_{M/_{\sim}}$ and the natural symplectic structure on each symplectic piece, $W_{\sigma,(H)}/_{\sim}$.
\end{Proposition}

\begin{Remark}
	This proposition shows that one may alternately view $\{,\}_{M/\sim}$ as defined point-wise on $M/_{\sim}$ by the symplectic structure on each symplectic piece.
\end{Remark}

\begin{proof}
	To show that $\{\cdot,\cdot\}_{M/_{\sim}}$ is a Poisson bracket on $C^{\infty}(M/_{\sim})$, we simply must show that $C^{\infty}(M)^{\sim}$ is a Poisson subalgebra of $C^{\infty}(M)$ (and therefore $\{\cdot,\cdot\}_{M/_{\sim}}$ is well-defined). 
	
	Let $f,g \in C^{\infty}(M)^{\sim}$ and let $W$ be one of the coisotropic submanifolds of $M$ as defined in \eqref{definition-of-W}.  Since $f$ and $g$ are constant on the leaves of the null foliation of $W$, we have that for all $Y\in T_wW^{\omega}$, 
	$$\omega(X_f,Y) = df(Y) = \mathcal{L}_Yf = 0$$
	and similarly for $g$. Thus $X_f,X_g \in (TW^{\omega})^{\omega} = TW$.
	
	Thus, for $Y\in T_wW^{\omega}$, we have that 
	\begin{equation*}
		\begin{split}
			\mathcal{L}_Y\{f,g\}_M & = \mathcal{L}_Y\omega(X_f,X_g)\\
				& = \iota_Y d(\omega(X_f,X_g))\\
				& = \iota_Y \iota_{[X_g,X_f]}\omega \\
				& = \omega([X_g,X_f],Y)
		\end{split}
	\end{equation*}
	which equals 0 since $[X_f,X_g]\in TW$. Thus $\{f,g\}_M \in C^{\infty}(M)^{\sim}$. Additionally, one sees that $fg \in C^{\infty}(M)^{\sim}$, so $C^{\infty}(M)^{\sim}$ is a Poisson subalgebra.

	To see that the inclusions of the symplectic pieces are Poisson, it is sufficient to observe that for all $f,g \in C^{\infty}(M)^{\sim}$, 
	$$\pi^*\iota^*\left\{f,g\right\}_{M/_{\sim}} = \left\{\pi^*f,\pi^*g\right\}_M\vert_{W_{\sigma,(H)}} = \pi^*\left\{\iota^*f,\iota^*g\right\}_{W_{\sigma,(H)}/_{\sim}}.$$
	See \cite[Theorem 25.3]{gs}.
\end{proof}

\begin{Remark}\label{rem:open dense symplectomorphism} If $\sigma_{\text{prin}}$ is the principal stratum of $\Delta$ corresponding to $(M,\omega,\mu)$. By part (3) of the symplectic cross-section theorem, the action of $K_{\sigma_{\text{prin}}}'$ on $\mu^{-1}(\sigma_{\text{prin}})$ is trivial, so $W_{\sigma_{\text{prin}},(K_{\sigma_{\text{prin}}}')} = K\cdot \mu^{-1}(\sigma_{\text{prin}})$. Since this is an open subset of $M$, its null foliation is trivial, so the restriction of $\pi$ is a symplectomorphism onto its image.	
\end{Remark}

\subsection{The $K\times T$ action on $M/_{\sim}$}

Since the leaves of the equivalence relation $\sim$ are invariant under the action of $K$ and $\mu$ is constant on these leaves, both the action of $K$ and the map $\mu$ descend to $M/_{\sim}$ to define a continuous function $\widetilde\mu$ and a continuous action of $K$. In terms of the diffeomorphism $W_{\omega,(H)}/_{\sim} \cong K\times_{K_{\sigma}}(Q_{\sigma,(H)}/K_{\sigma}')$, we have
$$k\cdot (k',[q]) = (kk',[q]) \text{ and } \widetilde\mu(k,[q]) = k\mu(q).$$
As observed in \cite{gsj}, the action of $T$ on $M$ leaves $\mu^{-1}(\Delta)$ invariant and (since $T$ normalizes each $K_{\sigma}'$) descends to a continuous action of $T$ on $\mu^{-1}(\Delta)/\sim$. This extends by $K$-equivariance to a continuous action of $T$ on $M/_{\sim}$ defined by the formula 
$$t\ast (k,[q]) = (k,t\cdot[q]) = (ktk^{-1}k,[q]) = (ktk^{-1})\cdot (k,[q])$$
which commutes with the action of $K$. Note that the action of $T$ on $\mu^{-1}(\Delta)$ does not, in general, extend to $M$ by $K$-equivariance.

It follows from the identifications established in the next section that the restriction of this $K\times T$ action to each symplectic piece of $M/_{\sim}$ is Hamiltonian, generated by the map $(\widetilde\mu,s\circ\widetilde\mu)$ where $s\colon\mathfrak{k}^* \to \Delta$ is the sweeping map of Thimm's trick (see \cite{hmm} for details).

\section{Symplectic contraction}\label{s:3}

In \cite{hmm}, HMM define the \emph{symplectic contraction} of a connected Hamiltonian $K$-manifold $(M,\omega,\mu)$ as the topological space 
$$ M^{sc} := \left(EM\times E_{\mathcal{L}}T^*K\right)\sslash_0 T$$
where $T$ is a choice of maximal torus of $K$, $EM$ is the symplectic implosion of $(M,\omega,\mu)$ with respect to a choice of positive Weyl chamber $\Delta \subseteq \mathfrak{t}^*$, and $E_{\mathcal{L}}T^*K$ is the symplectic implosion of $T^*K$ with respect to the cotangent lift of the left action of $K$ on itself and the opposite Weyl chamber $-\Delta$ (see \cite{hmm,gsj} for definitions).  This choice of $E_{\mathcal{L}}T^*K$ is isomorphic to $E_{\mathcal{R}}T^*K$ taken with respect to $-\Delta$ via the symplectic involution of $T^*K$, $(k,\lambda) \mapsto (k^{-1},-k\lambda)$, and thus we can rewrite the definition above as 
$$ M^{sc} := \left(EM\times E_{\mathcal{R}}T^*K\right)\sslash_0 T.$$
 The definition with $E_{\mathcal{R}}T^*K$ is preferable since $EM\times E_{\mathcal{R}}T^*K$ decomposes nicely into smooth pieces diffeomorphic to
	$$ (Q_{\sigma,(H)}/K_{\sigma}') \times K/K_{\sigma}'\times (-\sigma),$$ 
	where $Q_{\sigma,(H)}/K_{\sigma}'$ is the same as in the proof of Proposition \ref{prop:2}, equipped with symplectic form 
	$\widetilde \omega + \widetilde \Omega$ (cf. \cite{gsj}).  

The ``symplectic reduction'' by $T$ is taken with respect to the diagonal action of $T$ on $EM\times E_{\mathcal{R}}T^*K$ generated on symplectic strata by the  moment map $\mu+ \mu_{\mathcal{R},T}$ (or, equivalently, the diagonal action of $T$ on $EM\times E_{\mathcal{L}}T^*K$ generated by $\mu+ \mu_{\mathcal{L},T}$).  With respect to the description of the symplectic pieces given above, this moment map is explicitly given by the formula
$$(\mu+ \mu_{\mathcal{R},T})([q],kK_{\sigma}',\lambda) = \mu(q) + \lambda.$$
The diagonal action of $T$ is given explicitly by the formula
$$t\cdot ([q],kK_{\sigma}',\lambda) = ([t\cdot q],kt^{-1}K_{\sigma}',t\lambda).$$
Combining the facts above, one sees that the symplectic pieces of $M^{sc}$ are diffeomorphic to
$$ (Q_{\sigma,(H)}/K_{\sigma}')\times_{T}K/K_{\sigma}'.$$
We record some topological facts about symplectic contraction.

\begin{Proposition}\label{prop:sc basic topology}
	$M^{sc}$ is Hausdorff, second countable, locally compact, and connected.
\end{Proposition}
\begin{proof}
	Since $\mu + \mu_{\mathcal{R},T}$ is continuous, the level set $(\mu + \mu_{\mathcal{R},T})^{-1}(0)$ is closed in $EM \times E_{\mathcal{R}}T^*K$. By \cite[Theorem 2.3]{gsj} this implies that the level set is Hausdorff, locally compact, and second countable.
	
	The action of the compact group $T$ on the level set $(\mu + \mu_{\mathcal{R},T})^{-1}(0)$ is continuous, so the quotient map is open. Thus $M^{sc}$ is locally compact and second countable. Furthermore, any quotient of any locally compact Hausdorff space by a proper group action is Hausdorff, so $M^{sc}$ is Hausdorff.
	
	Finally, HMM prove that the symplectic contraction map $\Phi_{M}\colon M \rightarrow M^{sc}$ is continuous and surjective (see Proposition \ref{th:hmm contraction map properties} below), thus $M^{sc}$ is connected.
\end{proof}

HMM define the \emph{symplectic contraction map} $\Phi_M\colon M \rightarrow M^{sc}$ by the formula
	$$m \mapsto [hm,(h,\mu(m))] \in \left(EM\times E_{\mathcal{L}}T^*K\right)\sslash_0 T$$
where $h\in K$ such that $h\mu(m) \in \Delta$. HMM note that this map is well defined and $K$-equivariant: 
$$\Phi_M(k\cdot m) = [(hk^{-1})km,(hk^{-1},\mu(km))] = [(hm,(hk^{-1},k\mu(m))] = \mathcal{R}_k \Phi_M( m)$$
where the $K$ action on $M^{sc}$ descends from the right $K$ action on $E_{\mathcal{L}}T^*K$. Using the symplectic involution above, and writing $m=k\cdot q$ for $q\in \mu^{-1}(\Delta)$, the map $\Phi_M$ can be written equivalently as 
\begin{equation}\label{right-contraction-formula}
	k\cdot q \mapsto [q,(k,-\mu(q))] \in \left(EM\times E_{\mathcal{R}}T^*K\right)\sslash_0 T
\end{equation}
in which case $K$-equivariance is with respect the the descended left $K$ action on $E_{\mathcal{R}}T^*K$. HMM prove two main facts:


\begin{Proposition}\label{th:hmm contraction map properties}\cite{hmm} $\Phi_M$ is continuous, proper, and surjective.	
\end{Proposition}

\begin{Proposition}\label{prop:hmm contraction map symplectic}\cite{hmm} The restriction of $\Phi_M$ to the open dense set $\mu^{-1}(\Sigma_{\sigma_{\text{\emph{prin}}}})$ is a symplectomorphism onto its image.
\end{Proposition}

We observe that, combined with Proposition \ref{prop:sc basic topology}, Proposition \ref{th:hmm contraction map properties}  immediately implies the following.

\begin{Corollary}\label{cor:quotient map}
	$\Phi_{M}$ is a quotient map\footnote{Recall, a continuous map $f:X\to Y$ is a quotient map if it is surjective and a subset $U\subseteq Y$ is open iff $f^{-1}(U)$ is open.}.
\end{Corollary}

\begin{proof}
	Since $\Phi_{M}$ is proper and $M^{sc}$ is locally compact, $\Phi_M$ is closed. It follows since $\Phi_M$ is surjective that it is a quotient map.
\end{proof}

We can use the definition of $\Phi_M$ to describe its fibres.  For $k\in K_{\mu(m)}'$ and $h\in K$ such that $hk\cdot m\in \Delta$, 
\begin{equation*}
	\begin{split}
		\Phi_M(k\cdot m) & = [hk\cdot m,(h,\mu(k\cdot m))] \\
		& \sim [(hk^{-1}h^{-1})hk\cdot m,(h,\mu(m))] \\
		&  = [h\cdot m,(h,\mu(m))] \\
		& = \Phi_M(m)
	\end{split}	
\end{equation*}
so $\Phi_M$ is constant on the leaves of the stratified null foliation of $M$. Conversely, if $\Phi_M(m) = \Phi_M(m')$ then $\exists g_1,g_2 \in K_{\sigma}'$, $t\in T$ such that 
$$\left(tg_1hm,(tg_2h,\mu(m))\right) = \left(h'm',(h',\mu(m')) \right).$$
This implies that
$$ \mu(m) = \mu(m'), \, tg_2h=h',\, \text{ and } h^{-1}g_2^{-1}g_1hm=m'$$
which implies that $m$ and $m'$ lie in the same leaf of the stratified null foliation of $M$. Thus we conclude by Proposition \ref{prop:2} that,

\begin{Proposition}\label{prop:fibres} The fibres of $\Phi_M$ coincide with the leaves of the stratified null foliation of $M$.
\end{Proposition}

Thus, by Corollary \ref{cor:quotient map}, there exists a homeomorphism $\psi$ such that the diagram 
\begin{equation}\label{eq:diagram}
		\xymatrix{
			 && M/_{\sim} \ar[dd]^{\psi}\\  
			M\ar[rru]^{\pi}\ar[drr]_{\Phi_M}&&   \\
			&& M^{sc}} 
	\end{equation}
commutes. Since $\pi$ and $\Phi_M$ are both $K$-equivariant, $\psi$ is also equivariant (in fact it is $K\times T$ equivariant). This homeomorphism preserves the decompositions of the spaces $M^{sc}$ and $M/_{\sim}$ into pieces indexed by $\sigma\subseteq \Delta$ and $H \leq K_{\sigma}$. Moreover, the restriction of $\psi$ to each piece respects its smooth and symplectic structures:

\begin{Proposition}\label{symplectomorphism}
	The restriction of $\psi$ to each symplectic piece $W_{\sigma,(H)}/_{\sim}$ is a symplectomorphism onto its image (the corresponding symplectic piece in $M^{sc}$).
\end{Proposition}

\begin{proof}
	We have already seen that, considering $W_{\sigma,(H)}/_{\sim}$ with its quotient smooth structure, we have diffeomorphisms
	$$W_{\sigma,(H)}/_{\sim} \cong K\times_{K_{\sigma}}(Q_{\sigma,(H)}/_{K_{\sigma}'}) \cong (Q_{\sigma,(H)}/_{K_{\sigma}'})\times K/K_{\sigma} \times (-\sigma)\sslash_0 T$$
	given by the maps
	$$\pi(k\cdot q) \mapsto [k,q] \mapsto [q,(k,-\mu(q))]$$
	Since $\Phi_M(k\cdot q) = [q,(k,-\mu(q))]$ (cf. equation \eqref{right-contraction-formula}), we see that this composition of maps equals the restriction of $\psi$ to $W_{\sigma,(H)}/_{\sim}$.
	
	To see that the restriction of $\psi$ to $W_{\sigma,(H)}/_{\sim}$ is a symplectomorphism, it is sufficient to show that $$\pi^*\psi^*(\widetilde\omega|_{Q_{\sigma,(H)}}+\widetilde\Omega|_{K\times \mathfrak{S}_{\sigma}}) = \Phi_M^*(\widetilde\omega_{Q_{\sigma,(H)}}+\widetilde\Omega|_{K\times \mathfrak{S}_{\sigma}}) = \omega|_{W_{\sigma,(H)}}$$ 
	at a point $q\in Q_{\sigma,(H)}$. Here $\widetilde\Omega|_{K\times \mathfrak{S}_{\sigma}}$ is the symplectic form on the symplectic quotient $K/K_{\sigma}'\times (-\sigma) = \left(K \times \mathfrak{S}_{-\sigma}\right)\sslash_0 K_{\sigma}'$ (cf. \cite[p. 162]{gsj}). The result then follows by $K$-equivariance of $\psi$.
	 
	An arbitrary element of $T_{q}W_{\sigma,(H)}$ can be written as $X + \underline{Y}$ where $X \in T_qQ_{\sigma,(H)}$ and $\underline{Y}$ is the image of $Y\in \mathfrak{k}$ at $q$ under the Lie algebra action.  One computes that 
	$$(d\Phi_M)_{q}(X + \underline{Y}) = (\pi_*(X),Y,-d\mu_q(X))$$
	where we write $Y$ to mean $Y+\mathfrak{k}_{\sigma}' \in \mathfrak{k}/\mathfrak{k}_{\sigma}'=T_e (K/K_{\sigma}')$. For $X,X'\in T_qQ_{\sigma,(H)}$ and $Y,Y' \in \mathfrak{k}$, we compute
	\begin{equation*}
		\begin{split}
			&\Phi_M^*\left( \widetilde\omega|_{Q_{\sigma,(H)}}+\widetilde\Omega|_{K\times \mathfrak{S}_{\sigma}}\right)_{q}\left(X+\underline{Y},X' + \underline{Y'}\right)\\ 
			& = \left(\widetilde\omega|_{Q_{\sigma,(H)}}+\widetilde\Omega|_{K\times \mathfrak{S}_{\sigma}}\right)_{([q],e,-\mu(q))}\left((\pi_*X,Y,-d\mu_q(X)),(\pi_*X',Y',-d\mu_q(X'))\right)\\
			& = (\widetilde\omega|_{Q_{\sigma,(H)}})_{[q]}(\pi_*X,\pi_*X') \\
			& \qquad + \Omega_{(e,-\mu(q))}\left((Y,-d\mu_q(X)),(Y',-d\mu_q(X'))\right)\\
			& = \omega_q(X,X') + \langle -d\mu_q(X),Y'\rangle - \langle -d\mu_q(X'),Y\rangle - \langle -\mu(q),[Y,Y']\rangle\\
			& = \omega_q(X,X') + \omega_q(X,\underline{Y'}) + \omega_q(\underline{Y},X') + \omega_q(\underline{Y},\underline{Y'})\\
			& = \omega_{q}(X+\underline{Y},X'+\underline{Y'})
		\end{split}
	\end{equation*}
	where in the penultimate equality we have used Hamilton's equation and the fact that $\mu$ is Poisson. Thus the restriction of $\psi$ to $W_{\sigma,(H)}/_{\sim}$ is a symplectomorphism.
\end{proof}

	The homeomorphism $\psi$, along with the algebra $C^{\infty}(M/_{\sim})$ defined in Section 2, shows that $M^{sc}$ is endowed with a naturally defined algebra of smooth functions equipped with a Poisson bracket. This was not evident from the algebraic definition of HMM via symplectic implosion. Indeed, the symplectic implosion of a Hamiltonian $K$-manifold $(M,\omega,\mu)$ is not naturally a symplectic stratified space in the sense of \cite{sl} (cf. the comment in \cite{gsj} on page 167).  We end this section with the following observation.

\begin{Theorem}
	The symplectic contraction of a Hamiltonian $K$-manifold $(M,\omega,\mu)$ is a stratified space in the sense of \cite{sl}; the decomposition of $M^{sc}$ into symplectic pieces satisfies the following conditions.
	\begin{enumerate}[i)]
		\item (locally finite) Every point in $M^{sc}$ has a neighbourhood which intersects finitely many of the symplectic strata.
		\item (frontier condition) If for two symplectic strata $X$ and $Y$, $X\cap \overline{Y} \neq \emptyset$, then $X \subseteq \overline{Y}$.
		\item (local normal triviality) Every point in $M^{sc}$ has a neighbourhood homeomorphic to a cone over a lower dimensional stratified space.
	\end{enumerate}
	Moreover, $M^{sc}$ is a symplectic stratified space in the sense of \cite{sl}; the smooth structure $C^{\infty}(M^{sc})$ defined above satisfies the following conditions.
	\begin{enumerate}[a)]
		\item The strata are symplectic manifolds.
		\item $C^{\infty}(M^{sc})$ is a Poisson algebra.
		\item The inclusions of the strata are smooth Poisson maps.
	\end{enumerate}

\end{Theorem}

\begin{proof} Conditions i)-iii) follow from the corresponding facts for symplectic imploded spaces \cite{gsj} by HMM's definition of $M^{sc}$ as a reduction of imploded spaces.  Conditions a)-c) follow by Propositions \ref{poisson-structure} and \ref{symplectomorphism}.
\end{proof}

\section{Fibers of Gelfand-Zeitlin systems}

In this section, we apply our geometric perspective to describe the fibers of Gelfand-Zeitlin systems, which -- as was observed in \cite{hmm} -- can be constructed via contraction. 

Given a Hamiltonian $K_n$-manifold $(M,\omega, \mu)$, and a chain of group homomorphisms, 
\begin{equation*}\label{chain}
	K_{1} \xrightarrow{\phi_2} K_{2} \xrightarrow{\phi_3} \cdots \xrightarrow{\phi_n} K_n,
\end{equation*}
where the $K_i$ are connected, compact Lie groups with maximal tori $T_i$, we obtain a chain of contraction maps
\begin{equation}\label{iterated}
M \xrightarrow{\Phi_n} M_n \xrightarrow{\Phi_{n-1}} M_{n-1} \xrightarrow{\Phi_{n-2}} \cdots \xrightarrow{\Phi_1} M_1 
\end{equation}
in the following way.  First, by performing symplectic contraction with respect to the $K_n$ action on $M$, we get a symplectic contraction map $\Phi_n$ from  $M$ to the Hamiltonian $K_n\times T_n$-space $M_n$.  $M_n$ stratifies into symplectic manifolds equipped with a Hamiltonian $K_{n-1}\times T_n$ action coming from the homomorphism $\phi_n\colon K_{n-1}\to K_n$. Second, we take the quotient of $M_n$ by simultaneously performing symplectic contraction of all the symplectic strata of $M_n$ with respect to the $K_{n-1}$ action (note that since $T_n$ is abelian, this is the same as the symplectic contraction with respect to the $K_{n-1}\times T_n$ action). This results in a continuous map $\Phi_{n-1}\colon M_n \to M_{n-1}$, and $M_{n-1}$ is equipped with a $K_{n-1} \times T_{n-1}\times T_n$ action whose restriction to the symplectic pieces are again Hamiltonian. Repeating this procedure, one arrives at the space $M_1$, equipped with a Hamiltonian action of $T_1 \times \cdots \times T_{n-1}\times T_n$, generated by a moment map $\widetilde \mu$ such that the following diagram
\begin{equation}\label{cd}
		\xymatrix{
			M\ar[rr]^{\Phi}\ar[dd]^{\mu} && M_1\ar[dd]^{\widetilde{\mu}}\\  
			&&   \\
			\mathfrak{k}^* \ar[rr]^{F}&& \mathfrak{t}_1^*\times \cdots \times \mathfrak{t}_n^*
		} 
\end{equation}
commutes, where $F$ is the Gelfand-Zeitlin system on $\mathfrak{k}^*$ constructed from the chain of groups \eqref{chain} as in \cite{gs} and $\Phi = \Phi_1 \circ \cdots \circ \Phi_n$.   The space $M_1$ is the branching contraction space considered by HMM in \cite{hmm}.

Following work by \cite{thimm} and others on collective integrable systems, Guillemin and Sternberg observed in \cite{gs1,gs3} that given a multiplicity free Hamiltonian $K$ manifold for $K= U(n)$ or $SO(n)$, the Gelfand-Zeitlin system constructed from a chain of subgroups  
\begin{equation}\label{integrable-chains}
	U(1) \leq \cdots \leq U(n) \text{ or } SO(2) \leq \cdots \leq SO(n)
\end{equation}
defines a completely integrable torus action on the open dense subset of $M$ where the Gelfand-Zeitlin functions are smooth. 

We now show that, in general, if this construction yields a completely integrable system on an open dense subset of $M$, then the action of $T_1\times \cdots \times T_n$ on each symplectic piece of $M_1$, is completely integrable.

\begin{Lemma}\label{lemma:torus action is multiplicity free}
	If $(M,\omega,\mu)$ is a multiplicity free Hamiltonian $K$ manifold with connected fibers, then the action of the maximal torus $T$ on each of the symplectic pieces $Q_{\sigma,(H)}/K_{\sigma}' \subseteq EM$ is completely integrable.
\end{Lemma}

\begin{proof}
	The action of $K$ on $M$ is multiplicity free iff the reduced spaces $M\sslash_{\lambda}K$ are points for all $\lambda$ (cf. \cite[Proposition A.1]{woodward}). By \cite[Theorem 3.4]{gsj}, for $\lambda\in \sigma\subseteq \Delta$, 
	$$M\sslash_{\lambda}K \cong EM\sslash_{\lambda} T = \left(\bigcup_{H\leq K_{\sigma}'} Q_{\sigma,(H)}/K_{\sigma}'\right)\sslash_{\lambda} T $$
	is therefore a point. It follows that the action of $T$ on each symplectic piece $Q_{\sigma,(H)}/K_{\sigma}'$ is multiplicity free, or, in other words, completely integrable.
\end{proof}

\begin{Proposition}\label{st-action}
	 Let $(M,\omega,\mu)$ be a multiplicity free Hamiltonian $K$ manifold with connected fibers, and let $M^{sc}$ be its symplectic contraction.  Suppose that $S \leq K$ is a connected Lie subgroup such that the action of $S$ on every $K$ coadjoint orbit is multiplicity free. Then, every symplectic stratum of $M^{sc}$ is a multiplicity free Hamiltonian $S\times T$ manifold.
\end{Proposition}
\begin{proof}
	Every symplectic piece of $M^{sc}$ is of the form $K\times_{K_{\sigma}}(Q_{\sigma,(H)}/K_{\sigma}')$.  We want to show that the symplectic reduced spaces 
	$$\left(K\times_{K_{\sigma}}(Q_{\sigma,(H)}/K_{\sigma}')\right)\sslash_{(\xi,\lambda)}S\times T$$
	are all points.  By reduction in stages \cite{sl}, this space is isomorphic to 
	\begin{equation}\label{stages}
	\left(K\times_{K_{\sigma}}(Q_{\sigma,(H)}/K_{\sigma}')\sslash_{\lambda} T\right)\sslash_{\xi}S
	\end{equation}
	By Lemma \ref{lemma:torus action is multiplicity free}, the symplectic reduction $(Q_{\sigma,(H)}/K_{\sigma}')\sslash_{\lambda}T$ is a point. It follows that $$K\times_{K_{\sigma}}(Q_{\sigma,(H)}/K_{\sigma}')\sslash_{\lambda} T = K\times_{K_{\sigma}}\{\ast\} \cong K/K_{\sigma}.$$
	The Hamiltonian action of $K$ on $K\times_{K_{\sigma}}(Q_{\sigma,(H)}/K_{\sigma}')$ commutes with the action of $T$, so it descends to the symplectic quotient.  By $K$-equivariance, and the line above, the moment map for the action of $K$ on $K/K_{\sigma}$ is a symplectomorphism onto the $K$ coadjoint orbit through $\lambda$, so 
	$$K\times_{K_{\sigma}}(Q_{\sigma,(H)}/K_{\sigma}')\sslash_{\lambda} T \cong  K\cdot \lambda.$$
	By our assumption that all $K$ coadjoint orbits are multiplicity free Hamiltonian $S$ manifolds, we conclude that the space 
	$$\left(K\times_{K_{\sigma}}(Q_{\sigma,(H)}/K_{\sigma}')\sslash_{\lambda} T\right)\sslash_{\xi}S = \{\ast\},$$
	thus $K\times_{K_{\sigma}}(Q_{\sigma,(H)}/_{\sim})$ is a multiplicity free $S\times T$ manifold.
\end{proof}

We require the following fact (cf. \cite{gs3}).

\begin{Lemma}\label{mult-free-orbits}
	Every $U(n)$ coadjoint orbit is a multiplicity free $U(n-1)$ manifold for any embedding of $U(n-1)$ as a subgroup of $U(n)$. Respectively, every $SO(n)$ coadjoint orbit is a multiplicity free $SO(n-1)$ manifold for any embedding of $SO(n-1)$ as a subgroup.
\end{Lemma}

With these results in hand, we can conclude the following: every symplectic piece of the iterated symplectic contraction $M_1$ corresponding to a Gelfand-Zeitlin system is a toric manifold.

\begin{Theorem}\label{pieces-are-toric}
	Suppose $(M,\omega,\mu)$ is a multiplicity free Hamiltonian $U(n)$ or $SO(n)$ manifold. Let $M_1$ be the space constructed from $M$ by iterated symplectic contraction as in \eqref{iterated}, using one of the chains \eqref{integrable-chains}.  Then the action of $T = T_1 \times \cdots \times T_n$ on every symplectic piece of $M_1$ is completely integrable.
\end{Theorem}
\begin{proof}
	We apply Proposition \ref{st-action} at each stage of the iterated symplectic contraction for the case of $K=U(n)$ (the proof for $K=SO(n)$ is identical). 
	
	First, by Proposition \ref{st-action} and Lemma \ref{mult-free-orbits} we have that the symplectic pieces of $M_n$ are multiplicity free Hamiltonian $U(n-1) \times T_n$ manifolds. 
	
	If we apply Proposition \ref{st-action} and Lemma \ref{mult-free-orbits} again, to the symplectic pieces of $M_n$, it follows that the symplectic pieces of $M_{n-1}$ are multiplicity free $U(n-2)\times T_{n-1}\times T_n$ manifolds (note that we perform symplectic contraction with respect to the action of $U(n-1) \times T_n$, the maximal torus of which is $T_{n-1}\times T_n$, the result is identical to performing symplectic contraction with respect to the $U(n-1)$ action, except that this way the extra $T_n$ action descends as part of the construction). 
	
	Repeating this process, we finally have that the symplectic pieces of $M_1$ are multiplicity free $T_1 \times T_2 \times \cdots \times T_n$ manifolds (note that since $U(1) = T_1$, the last symplectic contraction map $\Phi_1$ is trivial, so $M_2=M_1$). In other words, the torus action on each symplectic piece is completely integrable. 
\end{proof}

This result allows us to give a very general description of the fibers of Gelfand-Zeitlin systems, similar to that of \cite{cko}.

\begin{Theorem}\label{gelfand-zeitlin-fibers}
	Suppose $(M,\omega,\mu)$ is a connected Hamiltonian $U(n)$ or $SO(n)$ manifold with $\mu$ proper, equipped with a completely integrable Gelfand-Zeitlin system constructed as above. Then the fibers of the Gelfand-Zeitlin system are the total spaces of sequences of fiber bundles
	\begin{equation}\label{iterated-bundle}
		E_n\rightarrow E_{n-1} \rightarrow \cdots \rightarrow E_2 \rightarrow E_1 = L
	\end{equation}
	where $L$ is an isotropic torus contained in symplectic piece of $M_1$ and the each
	$$(K_{k-1})_\sigma'/H_{k-1} \rightarrow E_k \rightarrow E_{k-1}$$ 
	is a fiber bundle of homogeneous spaces, where $(K_{k-1})_\sigma'$ is the commutator of a Levi subgroup of $K_{k-1}$ and $H_{k-1}$ is an isotropy subgroup (as in Section \ref{s:2}). 
\end{Theorem}

Note: it follows immediately from the description of symplectic contraction as the quotient by a stratified null foliation (and the fact that $L$ is isotropic), that each $E_k$ is isotropic in $M_k$.
\begin{proof}
	If the Gelfand-Zeitlin construction yields an integrable system, then $(M,\omega,\mu)$ is a multiplicity free Hamiltonian manifold \cite{gs3}, so we are in the setting of Theorem \ref{pieces-are-toric}. If $\mu$ is proper, then $F\circ \mu$ is proper, so by \cite[Theorem 1]{lane} the fibers of $F\circ \mu$ are all connected.  Since the maps $\Phi_k$ are all surjective (they are quotient maps), and the diagram \eqref{cd} commutes, it follows that the fibers of $\widetilde \mu$ are connected.
	\begin{itemize}
	\item Since \eqref{cd} commutes, the fibers of the Gelfand-Zeitlin system $F\circ \mu$ equal the fibers of the composition $\widetilde{\mu}\circ\Phi$.
	\item Since the torus actions generated by $\widetilde{\mu}$ on the symplectic pieces of $M_1$ are completely integrable, the fibers of the restriction of $\widetilde{\mu}$ to any symplectic piece of $M_1$ are isotropic tori.
	\item Since the fibers of $\widetilde{\mu}$ are connected, the intersection of any fiber of $\widetilde{\mu}$ with a symplectic piece of $M_1$ is closed in the symplectic piece, and the symplectic pieces of $M_1$ are locally closed in $M_1$, each fiber of $\widetilde\mu$ is contained in a single symplectic piece of $M_1$. 
	\item At each stage of the iterated symplectic contraction, the pre-image under $\Phi_k$ of a submanifold $N$ of a symplectic piece of $M_{k-1}$ indexed by $\sigma\subseteq \Delta_{k-1}$ and $(H)$ is a fiber bundle over $N$ whose fibers are the homogeneous spaces $(K_{k-1})_\sigma '/H_{k-1}$. 
\end{itemize}
	
\end{proof}

\begin{Remark} Let $M=\mathcal{O}_{\lambda}$ be a $U(n)$ coadjoint orbit and consider the Hamiltonian action of $U(n)$ generated by the inclusion $\iota\colon \mathcal{O}_{\lambda} \to \mathfrak{u}(n)^*$. The fibers of Gelfand-Zeitlin systems on $M$ were studied extensively by Cho-Kim-Oh \cite{cko} who prove a more detailed result analogous to Theorem \ref{gelfand-zeitlin-fibers} (note: the iterated fiber bundle structure of the fibers described in \cite{cko} is similar but not identical to the one described here).  Cho-Kim-Oh show that -- in this specific case, $M=\mathcal{O}_{\lambda}$ -- the only fibers occurring in the bundles $E_k \to E_{k-1}$ of \eqref{iterated-bundle} are points or odd-dimensional spheres.  Moreover -- in this specific case -- they show that the fibers of the Gelfand-Zeitlin system are all isotropic.
	
\end{Remark}

\bibliographystyle{unsrt}
\bibliography{bibliography}

\end{document}